\documentclass[12pt]{article}%
\usepackage{graphicx}
\usepackage{amsmath}
\usepackage{amsfonts}
\usepackage{amssymb}
\usepackage[usenames]{color}%
\providecommand{\U}[1]{\protect\rule{.1in}{.1in}}
\newtheorem{theorem}{Theorem}

\newtheorem{claim}[theorem]{Claim}

\newtheorem{lemma}[theorem]{Lemma}

\newtheorem{proposition}[theorem]{Proposition}

\newenvironment{proof}[1][Proof]{\noindent\textbf{#1.} }{\ \rule{0.5em}{0.5em}}
\begin{document}

\title{\textbf{Metastability of Queuing Networks with Mobile Servers\footnotetext{This
work was initiated in 2010 at the Newton Institute in Cambridge, UK,
as part of the semester on
Stochastic Networks. It then developed thanks to a CNRS France-Russia
international grant (Mobilit\'e, Dynamique et Th\'eorie des Files
d'Attente to 1,3 and 5). 
The core of the research was then carried out in France, Russia
and the USA: in France through grants to 1, in the framework of the Labex Archimede
(ANR-11-LABX-0033) and of the A*MIDEX project (ANR-11-IDEX-0001-02),
funded by the \textquotedblleft Investissements d'Avenir" French
Government programme managed by the French National Research Agency (ANR); in
Russia through a grant of the Russian
Foundation for Sciences (project No. 14-50-00150) to 3; in the USA
through a grant of the Simons 
Foundation (Grant No. 197982) to 2. All these supports are gratefully
acknowledged.}}}
\author{
F. Baccelli$^{2,5}$,
A. Rybko$^{3}$,
S. Shlosman$^{4,1,3}$ and
A. Vladimirov$^{3}$
\\{ }
\vspace{.2cm}
\\{\small $^{1}$Aix Marseille Universit\'{e}, Universit\'{e} de Toulon, 
CNRS, CPT, UMR 7332, France}
\\{\small $^{2}$UT Austin, Department of Mathematics, USA}
\\{\small $^{3}$Institute of the Information Transmission Problems, RAS, Moscow, Russia }
\\{\small $^{4}$Skolkovo Institute of Science and Technology, Moscow, Russia }
\\{\small $^{5}$INRIA Paris, France}
}
\maketitle

\begin{abstract}
We study symmetric queuing networks with moving servers and FIFO service
discipline. The mean-field limit dynamics demonstrates unexpected behavior
which we attribute to the metastability phenomenon. Large enough finite
symmetric networks on regular graphs are proved to be transient for
arbitrarily small inflow rates. However, the limiting non-linear Markov
process possesses at least two stationary solutions. The proof of transience
is based on martingale techniques.
\end{abstract}
\noindent
MCS: 60G10, 60G44, 60K25, 60K35, 60J25, 37A60. 

\section{Introduction}

In this paper we consider networks with moving servers. The setting is the
following: the network is living on a finite or countable graph $G=\left(
V,E\right)  $, at every node $v\in V$ of which one server $s$ is located at
any time. For every server, there are two incoming flows of customers: the
exogenous customers, who come from the outside, and the transit customers, who
come from some other servers. Every customer $c$ coming into the network
(through some initial server $s\left(  c\right)  $) is assigned a destination
$D\left(  c\right)  \in V$ according to some randomized rule. If a customer
$c$ is served by a server located at $v\in V,$ then it jumps to a server at
the node $v^{\prime}\in V,$ such that \textrm{dist}$\left(  v^{\prime
},D\left(  c\right)  \right)  =$\textrm{dist}$\left(  v,D\left(  c\right)
\right)  -1,$ thereby coming closer to its destination. If there are several
such $v^{\prime},$ one is chosen uniformly. There the customer $c$ waits {in
the FIFO queue} until his service starts. If a customer $c$ completes his
service by the server located at $v,$ and it so happens that \textrm{dist}%
$\left(  v,D\left(  c\right)  \right)  $ is $1$ or $0$, the customer is
declared to have reached its destination and leaves the network.

The important feature of our model is that the servers of our network are
themselves moving over the graph $G.$ Namely, we suppose that any two servers
$s,s^{\prime}$ located at adjacent nodes of $G$ exchange their positions as
the clock associated to the edge rings. The time intervals between the rings
of each alarm clock are i.i.d. exponential with rate $\beta$. When this
happens, each of the two servers brings all the customer waiting in its buffer
or being served, to the new location. In particular, it can happen that after
such a swap, the distance between the location of the customer $c$ and its
destination $D\left(  c\right)  $ increases (at most by one). We assume that
the service times of all customers at all servers are i.i.d. exponential with
rate $1$.

The motivation for this model comes from opportunistic multihop routing in
mobile wireless networks. Within this context, the servers represent mobile
wireless devices. Each device moves randomly on the graph $G$ which represents
the phase space of device locations. The random swaps represent the random
mobile motions on this phase space. Each node $v\in G$ of the phase space
generates an exogenous traffic (packetized information) with rate $\lambda
_{v}$ corresponding to the exogenous customers alluded to above. Each such
packet has some destination, which is some node of $G$. In opportunistic
routing (see Volume 2 of \cite{BB}), each wireless device adopts the following
simple routing policy: any given packet scheduled for wireless transmission is
sent to the neighboring node which is the closest to the packet destination.
The neighbor condition represents in a simplified way the wireless
constraints. It implies a multihop route in general. This routing policy is
the most natural one to use in view of the lack of knowledge of future random
swaps. For details on this motivation, the reader may refer to the literature
on mobile ad hoc networks and that on delay-tolerant networks. To the best of
our knowledge, the present paper is the first mathematical attempt to analyze
queuing within this mobile wireless framework.

In this paper we study the following types of graphs:

\begin{enumerate}
\item Finite graphs $G$: we consider the cyclic graph $C_{K}=\mathbb{Z}%
^{1}/K\mathbb{Z}^{1}$ and the toric graph $\mathcal{T}_{KL}=\mathbb{Z}%
^{2}/K\mathbb{Z}^{1}\oplus L\mathbb{Z}^{1},$ as well as general connected
$g$-regular graphs (i.e. graphs where every vertex has $g$ adjacent edges).

\item Mean-field graphs $G^{N}$: for a graph $G=\left(  V,E\right)  $ we
denote by $G^{N}$ a graph with the node set $V\times\left\{
1,2,...,N\right\}  ;$ two nodes $\left(  v^{\prime},n^{\prime}\right)  $ and
$\left(  v^{\prime\prime},n^{\prime\prime}\right)  $ are connected by an edge
iff $\left(  v^{\prime},v^{\prime\prime}\right)  \in E.$

\item Infinite graphs (like $\mathbb{Z}^{1}$ and $\mathbb{Z}^{2}$).

\item Limit graphs $G^{\infty}=\lim_{N\rightarrow\infty}G^{N}.$ The resulting
limiting networks can be analyzed using the theory of Non-Linear Markov
Processes (NLMP).
\end{enumerate}

The first two classes of graphs represent the type of wireless networks
alluded to above. The interest in mean-field versions of the last two types is
both of mathematical and practical nature. The mathematical interest of the
mean-field version of a network is well documented. There are also practical
motivations for analyzing such networks: their properties are crucial for
understanding the long-time behavior of finite size networks.

The results we obtain depend on the graph and look somewhat surprising. First
of all, we find that for finite regular graphs, the network is transient once
the diameter of the graph is large enough. For example, consider the network
on the graph $C_{K}$ with Poisson inflows with rate $\lambda>0$ at all nodes,
exponential service times with rate $1,$ FIFO discipline and node swap rate
$\beta>0.$ Then for all $K\geq K\left(  \lambda,\beta\right)  $, the queues at
all servers tend to infinity as time grows. In words this means that the
network is unstable for any $\lambda,$ however small it is -- once the network
is large enough. The same holds, probably, for the network on $\mathbb{Z}%
^{1},$ but we do not prove it here.

The same picture takes place for graphs $G^{N}$ with $N$ finite. However, the
limiting picture, for $N=\infty,$ is different: the corresponding NLM
processes on $C_{K}$ and $\mathbb{Z}^{1}$ have stationary distributions,
provided $0<\lambda\leq\lambda_{cr}\left(  K,\beta\right)  ,$ with
$0<\lambda_{cr}\left(  K,\beta\right)  <\infty$ for all $K\leq\infty.$
Moreover, for all $\lambda<\lambda_{cr}$ there are at least two different
stationary distributions, see Sect. \ref{FT} for more details.

On the other hand, the general convergence result of \cite{BRS} claims the
convergence of the networks on $G^{N}$ to the one on $G^{\infty}$ as
$N\rightarrow\infty,$ which seem to contradict to the statements above. The
explanation of this `contradiction' is that the convergence in \cite{BRS}
holds only on finite time intervals $\left[  0,T\right]  .$ That is, for any
$T$ there exists a value $N=N\left(  T\right)  $, such that the network on
$G^{N}$ is close to the limiting network on $G^{\infty}$ for all $t\in\left[
0,T\right]  ,$ provided $N\geq N\left(  T\right)  .$ Putting it differently,
the $G^{N}$ network behaves like the limiting $G^{\infty}$ network -- and
might even look as a stationary process -- for quite a long time, depending on
$N,$ but eventually it departs from such a regime and gets into the divergent
one. Clearly, the picture we have is an instance of metastable behavior. We
believe that more can be said about the metastable phase of our networks,
including the formation of critical regions of servers with oversized queues,
in the spirit of statistical mechanics, see e.g. \cite{SS}, but we will not
elaborate here on that topic.

\section{Finite networks}

{ This section starts with a detailed description of the methodology for
proving the instability of finite networks. This is done on the special case
of cyclic networks. We then discuss the extension to the mean-field versions
of cyclic networks and to toric networks. The general results, which bear on
connected $d$-regular graphs, are considered last. }

\subsection{The cyclic network}

\label{SS:cycn}

We start with the cyclic graph $C_{K}=\mathbb{Z}^{1}/K\mathbb{Z}^{1}.$ We use
the notation $C_{K}=(V_{K},E)$, where $V_{K}=\{1,\dots,K\}$ and
$E=\{(1,2),\dots,(K-1,K),(K,1)\}$. For future simplicity we take $K$ to be odd.

We study a continuous-time Markov process on a countable state $Q,$ related to
the graph $C_{K}$. Namely,
\[
Q=\{q^{v}:v\in V_{K}\}=(V_{K}^{\ast})^{V_{K}},
\]
where $V_{K}^{\ast}$ is the set of all finite words in the alphabet $V_{K}$,
including the empty word $\varnothing$.

The queue $q^{v}\in V_{K}^{\ast}$ at a server located at $v\in V_{K}$ consists
of a finite $\left(  \geq0\right)  $ number of customers which are ordered by
their arrival times (FIFO service discipline) and are marked by their
destinations which are vertices of the graph $C_{K}$. Since the destination of
the customer is its only relevant feature, in our notations we sometime will
identify the customers with their destinations.

\subsubsection{Dynamics}

\label{ssdyn}

Let us introduce the continuous-time Markov process ${\mathcal{M=M}}\left(
t\right)  $ with the state space $Q$. Let $h_{v}$ be the length of the queue
$q^{v}$ at node $v$. We have $q^{v}=\left\{  q_{1}^{v},\dots,q_{h_{v}}%
^{v}\right\}  $ if $h_{v}>0$ and $q^{v}=\varnothing$ if $h_{v}=0$.

The following events may happen in the process ${\mathcal{M}}$.

An arrival event at node $v$ changes the queue at this node. If the newly
arrived customer has node $w$ for its destination, then the queue changes from
$q^{v}$ to $q^{v}\oplus w$, that is, to $\left\{  q_{1}^{v},\dots,q_{h_{v}%
}^{v},w\right\}  $ if $h_{v}>0$ or from $\varnothing$ to $\left\{  w\right\}
$ if $h_{v}=0$.

In this paper we consider the situation where each exterior customer acquires
its destination at the moment of first arrival to the system, in a
translation-invariant manner: the probability to get destination $w$ while
arriving to the network at node $v$ depends only on $w-v$ mod $K$. The case
$w=v$ is not excluded. We thus have the rates $\lambda_{v,w},$ $v,w\in C_{K},$
where $\lambda_{v,w},$ depends only on $w-v$ mod $K$, and the jump from
$q^{v}$ to $q^{v}\oplus w,$ corresponding to the arrival to $v$ of the
exterior customer with final destination $w$ happens with the rate
$\lambda_{v,w}.$ We introduce the rate $\lambda_{v}$ of exterior customers to
queue $v$ as
\begin{equation}
\lambda_{v}=\sum_{w}\lambda_{v,w}. \label{21}%
\end{equation}
According to our definitions, $\lambda_{v}=\lambda$ does not depend on $v$.

Each node is equipped with an independent Poisson clock with parameter $1$
(the service rate). As it rings, the service of customer $q_{1}^{v}$ is
completed, provided $h_{v}>0$; nothing happens if $h_{v}=0$. In the former
case the queue at node $v$ changes from $q^{v}$ to
\[
q_{-}^{v}=\left\{  q_{2}^{v},\dots,q_{h_{v}}^{v}\right\}
\]
(we also define $\varnothing_{-}=\varnothing$) and immediately one of the two
things happen: either the customer $q_{1}^{v}$ leaves the network, or it jumps
to one of the two neighboring queues, $q^{v\pm1}$. The customer $q_{1}^{v}$
leaves the network only if its current position, $v,$ is at distance $\leq1$
from its destination, i.e. iff $q_{1}^{v}=v-1,$ $v,$ or $v+1.$ This is just
one of many possible choices we make for simplicity. Otherwise it jumps to the
neighboring vertices $w=v\pm1,$ which is the closest to its destination, i.e.
to the one which satisfy: $\mathrm{dist}(w,q_{1}^{v})=\mathrm{dist}%
(v,q_{1}^{v})-1$ (there is a unique such $w\in V_{K}$ since we assume $K$ to
be odd. The case of even $K$ requires small changes).

The last type of events is the swap of two neighboring servers. Namely, there
is an independent Poisson clock at each edge $uv\in E$ of $C_{K},$ with rate
$\beta>0$. As it rings, the queues at the vertices $u$ and $v$ swap their
positions, that is,
\[
q^{v}(t_{+})=q^{u}(t),\qquad q^{u}(t_{+})=q^{v}(t).
\]

\subsubsection{Submartingales}

Here we introduce some martingale technique that will be used for the proof of
transience of ${\mathcal{M}}$ for $K$ large enough. To begin with, we label
the $K$ servers by the index $k=1,...,K$; this labelling will not change
during the evolution. Together with the original continuous-time Markov
process ${\mathcal{M}}(t)$ we will consider the embedded discrete time process
$M(n)$, which is the value of ${\mathcal{M}}(t)$ immediately after the $n$-th
event. The state of the process $M$ consists of the states of all $K$ servers
and all their locations.

The general theorem below will be applied to the quantities $X_{n}^{k},$ which
are, roughly speaking, the lengths of the queues at the servers $k,$
$k=1,...,K,$ of the process $M(\Lambda n)$. The integer parameter
$\Lambda=\Lambda(K,\lambda,\beta)$ will be chosen large enough, so that, in
particular, after time $\Lambda$, the locations of the servers are well mixed
on the graph $C_{K},$ and the joint distribution of their location on $V_{K}$
is {close to uniform on the set of permutations on $V_{K}$}. Below, we aim to
prove that the conditional expectations of all the differences $X_{n+1}%
^{k}-X_{n}^{k}$ are positive. We start with the following theorem.

\begin{theorem}
\label{T1} Let ${\mathcal{F}}={\mathcal{F}}_{n}$, $n=0,1,\dots$, be a
filtration and let $X_{n}^{k}$, $k=1,\dots,K$, be a finite family of
non-negative integer-valued submartingales adapted to ${\mathcal{F}},$ such
that for all $k=1,\dots,K$, and all $n=0,1,\dots$, the following assumptions hold:

\begin{enumerate}
\item[(1)] For some $\rho>0$ the inequality
\begin{equation}
{\mathbb{E}}_{\mathcal{F}_{n}}(X_{n+1}^{k}-X_{n}^{k})\geq\rho\label{01}%
\end{equation}
holds whenever $X_{n}^{k}>0$.

\item[(2)] The increments are bounded by a constant $R$:
\begin{equation}
|X_{n+1}^{k}-X_{n}^{k}|\leq R\quad\text{a.s.} \label{02}%
\end{equation}

\end{enumerate}

Then there exists an initial state $(X_{0}^{1},\dots, X_{0}^{K})$ such that,
with positive probability, $X_{n}^{k}\rightarrow+\infty$ as $n\rightarrow
+\infty$ for all $k=1,\dots,K.$
\end{theorem}

In order to prove the theorem we begin with an auxiliary lemma.

\begin{lemma}
\label{L1} Let $\mathcal{Y}^{k}=\{Y_{n}^{k}:n=0,1,\dots$, $k=1,\dots,K$, be a
finite family of submartingales adapted to the same filtration ${\mathcal{F}}$
and such that $Y_{n}^{k}\in\lbrack0,1]$ for all $k,n$. Suppose also that for
any $\varepsilon>0$ there exists a $\delta>0$ such that, for all $k$ and $n$,
{
\[
{\mathbb{E}}_{\mathcal{F}_{n}}(Y_{n+1}^{k}-Y_{n}^{k})>\delta\quad
\text{on}\quad0<Y_{n}^{k}<1-\varepsilon.
\]
} Suppose that the initial vector $Y_{0}\in A=[0,1]^{K}$ is deterministic and
satisfies the condition
\begin{equation}
\sum_{k=1}^{K}Y_{0}^{k}>K-1. \label{E1}%
\end{equation}
Then, with positive probability, $Y_{n}^{k}\rightarrow1$ as $n\rightarrow
\infty,$ for all $k=1,\dots,K$.
\end{lemma}

\begin{proof}
Since each submartingale $Y_{n}^{k}$ is bounded, there is a limit
$\lim_{n\rightarrow\infty}Y_{n}^{k}=Y^{k}$ almost surely for all $k$, see the
Martingale Convergence Theorem in \cite{D}. { Let us first show that the limit
vector $Y=(Y^{k})$ has its support on the union of the `maximal' vertex
$(1,\dots,1)$ of the cube $A$ and the `lower boundary' $B$ of $A$: $B=\left\{
a:\min_{k=1,\dots,K}a_{k}=0\right\}  $. Indeed, if $Y$ had parts of its
support on the complement $C$ of this union in $A$, then there would exist a
$k$ and $0<\alpha<\beta<1$ such that, $P(Y^{k}\in(\alpha,\beta))>0$. This
would imply that for some $\rho>0$,
\begin{align*}
{\mathbb{E}}\left[  (Y_{n+1}^{k}-Y_{n}^{k})\mathbb{I}_{\left\{  Y_{n}^{k}%
\in(\alpha,\beta)\right\}  }\right]   &  = {\mathbb{E}}\left[  {\mathbb{E}%
}_{{\mathcal{F}}_{n}} \left[  (Y_{n+1}^{k}-Y_{n}^{k})\mathbb{I}_{\left\{
Y_{n}^{k}\in(\alpha,\beta)\right\}  }\right]  \right] \\
&  \ge{\mathbb{E}}\left[  {\mathbb{E}}_{{\mathcal{F}}_{n}} \left[
\rho\mathbb{I}_{\left\{  Y_{n}^{k}\in(\alpha,\beta)\right\}  }\right]  \right]
\\
&  = \rho{\mathbb{P}} (Y_{n}^{k}\in(\alpha,\beta)) \to_{n\to\infty} a>0.
\end{align*}
But this contradicts the fact that
\[
\lim_{n\to\infty} {\mathbb{E}} \left[  (Y_{n+1}^{k}-Y_{n}^{k})\mathbb{I}%
_{\left\{  Y_{n}^{k}\in(\alpha,\beta)\right\}  }\right]  =0,
\]
which follows from the convergence a.s. of $Y_{n}^{k}$ and the dominated
convergence theorem.}

To conclude the proof of the lemma, note that for any random vector
$V=(V^{k})$ with support in $B$,
\begin{equation}
\sum_{k=1}^{K}V^{k}\leq K-1. \label{E2}%
\end{equation}
But, by the submartingale property, for all $n=1,\ldots$,
\begin{equation}
{\mathbb{E}}\sum_{k=1}^{K}Y_{n}^{k}\geq\sum_{k=1}^{K}Y_{0}^{k}>K-1. \label{E3}%
\end{equation}
Inequalities (\ref{E1})-(\ref{E3}) rule out the option that $Y$ has its
support in $B$.
\end{proof}

Now, in order to derive Theorem \ref{T1} from Lemma \ref{L1}, we make the
following change of variables for submartingales $X_{n}^{k}$. For a positive
parameter $\alpha<1$ we define an `irregular lattice' $h_{i}\in{\mathbb{R}%
}_{+}$, by
\[
h_{0}=0,\ \ h_{i+1}=h_{i}+\alpha^{i},\quad i=0,1,\dots.
\]
We get $\lim_{i\rightarrow\infty}h_{i}=H=(1-\alpha)^{-1}<\infty$. Now, for
each $k=1,\dots,K$, we define the process $Y_{n}^{k}$ on the same filtration
${\mathcal{F}}$ by the relation
\[
Y_{n}^{k}(\omega)=h_{X_{n}^{k}(\omega)}.
\]
For $k=1,\dots,K$, the process $Y_{n}^{k}$ takes its values on the `lattice'
$\{h_{i}\}$. { It is easy to see that under the assumptions that $|X_{n+1}%
^{k}-X_{n}^{k}|\leq R$ and that ${\mathbb{E}}_{\mathcal{F}_{n}}(X_{n+1}%
^{k}-X_{n}^{k})\geq\rho$, for each $k$, $Y_{n}^{k}$ is still a submartingale,
provided $1-\alpha$ is small enough.}
Then the hypothesis of Lemma \ref{L1} holds (up to a constant factor $H$) and
Theorem \ref{T1} is proved.

\subsubsection{Transience}

Let us return to the process ${\mathcal{M}}\left(  t\right)  $. Suppose that
the parameters $\lambda>0$ and $\beta>0$ are fixed. We remind the reader that
our service rate is set to $1$.

\begin{theorem}
\label{T3} For each $\lambda>0$ and $\beta>0$, there exists $K^{\ast}%
\in{\mathbb{Z}}_{+}$ such that for any $K\geq K^{\ast}$ the process
${\mathcal{M}}$ is transient.
\end{theorem}

{ The proof consists of the following steps:} first of all, we construct a
discrete time Markov chain ${\mathcal{D}}$ on the state space $Q$. To define
this chain, we start with the embedded Markov chain $M(n)$, defined earlier,
and then pass to the chain $M(\Lambda n)$, with the integer $\Lambda$ to be
specified later. To get the chain ${\mathcal{D=}}\left\{  {\mathcal{D}}%
_{n}\right\}  $, we modify the chain $M(\Lambda n)$ as follows: if for some
$n$ some of the $K$ queues are {less than or equal to $\Lambda$, we make all
queues to be of length exactly $\Lambda$, and then freeze the process at this
point forever.} Otherwise we do no changes.
We start the process ${\mathcal{D}}$ at some configuration $Q_{0}$ with all
queues longer than $\Lambda$. We then prove the following statement: if
$\Lambda$ is large enough, the process $\mathcal{D}$ at any given server $k$
is a submartingale satisfying the conditions of Theorem \ref{T1}, with respect
to the filtration defined by the discrete-time Markov chain $M(n)$ (the
individual queue length processes are clearly adapted to this filtration).
This will then complete the proof thanks to Theorem \ref{T1}.

{ These steps are embodied in a series of lemmas.}

Consider the following function $\pi(t)$ of the process ${\mathcal{M}}$
defined in Subsection \ref{ssdyn}. At each $t\ge0$, let $\pi(t)$ be the
current permutation of indices of the $K$ servers with respect to the $K$
nodes and let $v(i,t)$ denote the current position of server $i$.

\begin{lemma}
\label{L4} As $t\to\infty$,

\begin{enumerate}
\item {(1)} the distribution of $\pi(t)$ converges to the uniform distribution
on the set $S_{K}$ of all permutations;

\item {(2)} the distribution of $v(i,t)$ converges to the uniform distribution
on $\{1,\dots,K\}$ as $t\to\infty$.
\end{enumerate}
\end{lemma}

\begin{proof}
{ Note that $\{\pi(t)\}$ is a continuous time Markov process on its own. To
best represent this Markov process,} let us introduce a graph structure on the
permutation group $S_{K}$. Consider all the transpositions $\tau\in S_{K}$
corresponding to the exchanges of pairs of neighboring servers. We connect two
permutations $\pi^{\prime},\pi^{\prime\prime}$ by an edge iff $\pi^{\prime
}=\pi^{\prime\prime}\tau$ for some $\tau.$

The resulting graph on $S_{K}$ is connected -- because $G_{K}$ is connected {
and each of its nodes has the same degree.} The process of migration of
servers is a uniform random walk on this graph, that is, a reversible process.
Hence, as $t\rightarrow\infty$, the distribution of permutations converges to
the uniform one, for all initial states. The assertions of the lemma follow.
\end{proof}

\begin{lemma}
\label{L11} For all initial states $Q_{0}$, { the probability that a customer
with position $H>0$ in some queue leaves the network after being served at
this queue} tends to $3/K$ as $H\rightarrow\infty$, uniformly in $Q_{0}.$
\end{lemma}

\begin{proof}
As the waiting time of the customer tends to infinity with $H\rightarrow
\infty$, the distribution of its server on $V_{K}$ tends to the uniform one on
$C_{K}$ (see Lemma \ref{L4}). In order for the customer $c$ to exit the
network, the last server visited by $c$ has to be located at this moment at
one of the three nodes: $D(c) +1,$ $D(c)$ or $D(c)-1$.
\end{proof}

Hence, for all customers in the initial queues whose positions are at least
$H$, the mean chance of exit approaches $3/K$ as $H\rightarrow\infty$ and the
rate of this approach does not depend on the particularities of the initial
state $Q_{0}$, but only on $H$.

The next remark is that if a customer with initial position $H$, with $H$
large, is served and then jumps to a different server, then the index $j$ of
that server is distributed almost uniformly over the remaining $K-1$ indices.
This fact follows from Lemma \ref{L11}. Again, the rate of convergence is
independent of $Q_{0}$ because the servers swap their positions independently
of anything else. So we have established:

\begin{lemma}
\label{L12} The probability that a customer with position $H$ on server $i$
jumps to server $j$ at the completion of its service in this server tends to
$1/\left(  K-1\right)  $ as $H\rightarrow\infty$, uniformly in $i\ne j$ and in
the initial states $Q_{0}\in Q$.
\end{lemma}

We need a third combinatorial lemma. We start with some definitions before
formulating and proving this lemme. Let $\left\{  u,v\right\}  $ be an ordered
pair of elements of $V_{K}$. We define the map $T$ from the set of all such
pairs to $V_{K}\cup\left\{  \ast\right\}  ,$ by
\[
T\left\{  u,v\right\}  =\left\{
\begin{array}
[c]{cc}%
w &
\begin{array}
[c]{c}%
\text{for }w\text{ defined by }\left\vert u-w\right\vert =1,\ \left\vert
v-w\right\vert =\left\vert u-v\right\vert -1,\text{ }\\
\text{provided }\left\vert u-v\right\vert >1,
\end{array}
\\
\ast & \text{otherwise.}%
\end{array}
\right.
\]
For $K$ odd, the map $T$ is well-defined. In case $T\left\{  u,v\right\}  =w$,
we say that a customer transits through $w$ on his way from $u$ to $v$.

Let $D:V_{K}\rightarrow V_{K}$ be an arbitrary map. We want to compute the
quantity%
\begin{equation}
p_{K}=\frac{1}{K!}\sum_{\pi\in S_{K},i\in V_{K}}\mathbb{I}_{\left\{  T\left\{
\pi\left(  i\right)  ,D\left(  i\right)  \right\}  =\pi\left(  j\right)
\right\}  }, \label{331}%
\end{equation}
where $S_{K}$ is the symmetric group, while $j$ is an element of $V_{K}.$ That
is, $p_{K}$ is the continuous time transit rate through server $j$ in
stationary regime -- when the locations of servers are uniformly distributed
on $S_{K}$ -- provided all servers have infinite queues. Of course, this rate
does not depend on $j$.

\begin{lemma}
\label{L31}%
\[
p_{K}=\frac{K-3}{K}.
\]

\end{lemma}

\begin{proof}
Without loss of generality we can take $j=1$. Instead of performing the
summation in $\left(  \ref{331}\right)  $ over whole group $S_{K},$ we
partition $S_{K}$ into $\left(  K-2\right)  !$ subsets $A_{\pi}$, and perform
the summation over each $A_{\pi}$ separately. If the result will not depent on
$\pi,$ we are done. Here $\pi\in S_{K},$ and, needless to say, for $\pi
,\pi^{\prime}$ different we have either $A_{\pi}=A_{\pi^{\prime}}$ or $A_{\pi
}\cap A_{\pi^{\prime}}=\varnothing.$

Let us describe the elements of the partition $\left\{  A_{\pi}\right\}  .$ So
let $\pi$ is given, and the string $i_{1},i_{2},i_{3},...,i_{l},i_{l+1}%
,...,i_{K}$ is the result of applying the permutation $\pi$ to the string
$1,2,...,K.$ Then we include into $A_{\pi}$ the permutation $\pi,$ and also
$K-1$ other permutations, which correspond to the cyclic permutations, e.g. we
add to $A_{\pi}$ the strings $i_{K},i_{1},i_{2},i_{3},...,i_{l},i_{l+1},...,$
$i_{K-1},i_{K},i_{1},i_{2},i_{3},...,i_{l},i_{l+1},...,$ and so on. We call
these transformations `cyclic moves'. Now with each of $K$ permutations
already listed we include into $A_{\pi}$ also $K-2$ other permutations, where
the element $i_{1}$ does not move, and the rest of the elements is permuted
cyclically, i.e., for example from $i_{K-1},i_{K},i_{1},i_{2},i_{3}%
,...,i_{l},i_{l+1},...,$ we get $i_{K},i_{2},i_{1},i_{3},...,i_{l}%
,i_{l+1},...,i_{K-1},$ $i_{2},i_{3},i_{1},...,i_{l},i_{l+1},...,i_{K-1}%
,i_{K},$ and so on. We call these transformations `restricted cyclic moves'.
The main property of thus defined classes of configurations is the following:
Let $a\neq b\in\left\{  1,2,...,K\right\}  $ be two arbitrary indices, and
$l\in\left\{  2,...,K\right\}  $ be an arbitrary index, different from $1.$
Then in every class $A_{\pi}$ there exists exactly one permutation
$\pi^{\prime},$ for which $i_{1}=a$ and $i_{l}=b.$

Given $\pi,$ take the customer $l\neq1\left(  =j\right)  ,$ and its
destination, $D\left(  l\right)  .$ If we already know the position $i_{1}$ of
customer $1$ on the circle $C_{K},$ then in the class $A_{\pi}$ there are
exactly $K-1$ elements, each of them corresponds to a different position of
the server $l$ on $C_{K}.$ If it so happens that $i_{1}=D\left(  l\right)  ,$
then for no position of the server $l$ the transit from $l$ through $i_{1}$
happens. The same also holds if $i_{1}=\left(  D\left(  l\right)  +\frac
{K-1}{2}\right)  \operatorname{mod}K$ or $i_{1}=\left(  D\left(  l\right)
+\frac{K+1}{2}\right)  \operatorname{mod}K.$ For all other $K-3$ values of
$i_{1}$ the transit from $l$ through $i_{1}$ happens precisely for one
position of $l$ (among $K-1$ possibilities). Totally, within $A_{\pi}$ we have
$\left(  K-1\right)  \left(  K-3\right)  $ transit events. Since $\left\vert
A_{\pi}\right\vert =K\left(  K-1\right)  ,$ the lemma follows.

\end{proof}

\vspace{.3cm}

\noindent\textbf{Proof of Theorem \ref{T3}.} { We recall that $M_{n}^{k}$
denotes the queue length of server $k$ at the $n$-th transition of the chain
$M(t)$ and that $\{{\mathcal{F}}_{n}\}$ denotes the natural filtration of the
chain $\{M(n)\}$.}

We now define the submartingales $\{X_{n}^{k}\}$ and show that they satisfy
all the properties of Theorem \ref{T1}.

{ Let $\Lambda$ be a positive integer and for all $k$, let
\[
X_{n}^{k} = M_{n \Lambda}^{k}- \Lambda,\quad n=0,1,\ldots
\]
as long as the R.H.S. is positive, and $0$ from the first $n$ such that it is
less than or equal to 0 (see above). Clearly, $X_{n}^{k}\geq0$. We now show
that if $K$ and $\Lambda$ are both suitably large, then, for all $k$,
$\{X_{n}^{k}\}$ is a submartingale w.r.t. the filtration ${\mathcal{G}}%
_{n}={\mathcal{F}}_{\Lambda n}$, and in addition, Properties (1) and (2) of
Theorem \ref{T1} hold.}

{ On $X_{n}^{k}=0$, the submartingale property ${\mathbb{E}}_{{\mathcal{G}%
}_{n}} X_{n+1}^{k} \ge X_{n}^{k}$ is satisfied as $X_{n}^{k}=0$ implies
$X_{n+1}^{k}=0$.}

Relation (\ref{02}) is evidently satisfied with $R=\Lambda.$ Let us now check
(\ref{01}). Let us start the process $M$ at a configuration where all the
queue lengths are of the form $X_{0}^{k}+\Lambda$ with $X_{0}^{k}> 0$,
$k=1,...,K$. We want to show that { ${\mathbb{E}}_{{\mathcal{F}}_{0}%
}(M_{\Lambda}^{k}-M_{0}^{k})\geq\rho$, for some $\rho>0$, which will prove
(\ref{01}) and the submartingale property on $X_{n}^{k}>0$.} Let $H=H\left(
K\right)  $ be the time after which the distribution of the $K$ servers is
almost uniform on $C_{K},$ see Lemma \ref{L11}. Before this moment, we do not
know much about our network; we can nevertheless bound the lengths of the
queues $M_{H}^{k}$ from below by $M_{H}^{k}\geq M_{0}^{k}-H.$ After time $H$,
the probability that a customer leaving a server leaves the network is almost
$1/K,$ and the probabilities that it jumps to the left or the right are both
close to $\frac{K-1}{2K}.$ More precisely, by Lemma \ref{L31}, if $\Lambda$ is
large enough, after time $H$, the rate of arrival to every server is
approximately $\lambda+(K-3)/K,$ which is higher than the exit rate, $1,$
provided $K$ is large enough (namely, $K>K^{*}=3/\lambda$). Hence the expected
queue lengths in the process $M$ grow linearly in time, at least after time
$H$, which implies the existence of $\Lambda>0$ such that { $\mathbb{E}%
_{{\mathcal{F}}_{0}}\left(  M_{\Lambda}^{k}\right)  \geq M_{0}^{k}+\rho.$} So,
Theorem \ref{T1} applies. This completes the proof of Theorem \ref{T3}.
\quad$\blacksquare$

\subsection{{Mean-field version of the cyclic network}}

In this subsection, we analyze the mean field graph $C_{K}^{N}$ defined in the
introduction, where at each vertex $v\in C_{K}$ we now have $N$ servers. The
dynamics of the system is a modification of the case $N=1$ with the following
characteristics: the

\begin{itemize}
\item exogenous customers arrive to each node $(v,n)$ $v=1,...,K,$ $n=1,...,N$
with the rate $\lambda$;

\item the destination of an exogenous customer is a node $w$ in $C_{K}$ (and
not a node in $C_{K}^{N}$);

\item if a customer $c$ completes its service at a node $(v,n)$, then it
leaves the network in case its destination $D(c)$ is $v$ or $v\pm1;$ otherwise
it transits to the node $(\tilde{v},k)$, where $\tilde{v}$ is the node which
is the closest to $D(c)$ among $v+1$ and $v-1$, while $k$ is chosen uniformly
from the $N$ values $1,...,N$;

\item the two servers at locations $(v,n)$ and $(v+1,k)$ swap with rate
$\frac{\beta}{N}.$
\end{itemize}

The results in this case, as well as the proofs, are similar to those of
Subsection \ref{SS:cycn}, except for the analogue of Lemma \ref{L31}. Below we
define the corresponding ensemble and we formulate the analogous statement.

Two nodes $\left(  v,k\right)  $ and $\left(  v^{\prime},k^{\prime}\right)  $
of $C_{K}^{N}$ are connected by an edge iff $v$ and $v^{\prime}$ are connected
by an edge in $G$. Let $\left\{  \left(  u,k\right)  ,\left(  v,l\right)
\right\}  $ be an ordered pair of nodes of $C^{N}_{K}$. We define the
\textit{random }map $T$ from the set of all such pairs into the union
$V_{K}\times N\cup\left\{  \ast\right\}  ,$ by%
\[
T\left\{  \left(  u,k\right)  ,\left(  v,l\right)  \right\}  =\left\{
\begin{array}
[c]{cc}%
\left(  w,m\right)  \text{ with probability }\frac{1}{N} &
\begin{array}
[c]{c}%
\text{if }\left\vert u-v\right\vert >1,\\
w\text{ satisfies }\left\vert u-w\right\vert =1,\ \\
\left\vert v-w\right\vert =\left\vert u-v\right\vert -1,\text{ }%
\end{array}
\\
\ast & \text{otherwise.}%
\end{array}
\right.
\]
For $K$ odd, the map $T$ is well-defined. In case $T\left\{  \left(
u,k\right)  ,\left(  v,l\right)  \right\}  =\left(  w,m\right)  $ we say that
we have a transit of a customer through $\left(  w,m\right)  .$

We define a subgroup $\tilde{S}_{K}\subset S_{KN}$ of permutations of the
nodes of the graph $G\times N$ as the one generated by the transpositions
$\left(  u,k\right)  \leftrightarrow\left(  v,l\right)  ,$ where
$u\leftrightarrow v$ is a transposition from $S_{K},$ and $k,l$ are arbitrary.

Let $D:V_{K}\times N\rightarrow V_{K}$ be an arbitrary map. The analogue of
$p_{K}$ in (\ref{331}) is the quantity {
\[
p_{K,N}=\frac{1}{\left\vert \tilde{S}_{K}\right\vert } \sum_{\pi\in\tilde
{S}_{K},\left(  u,k\right)  \in V_{K}\times N} \sum_{(u,k)} \mathbb{P}\left(
T\left\{  \pi\left(  u,k\right)  ,D\left(  u,k\right)  \right\}  = \pi\left(
w,m\right)  \right)  .
\]
} By arguments similar to those of the last subsection, it is easy to show
that $p_{K,N}=\frac{K-3}{K}.$ Hence, the following theorem holds:

\begin{theorem}
For all $\lambda>0$, $\beta>0$ and $N\ge1$, and for all $K> K^{\ast}%
=3/\lambda$, the Markov process ${\mathcal{M}_{K,N}}$ is transient.
\end{theorem}

\subsection{{The toric network}}

In this subsection, the graph is $\mathcal{T}_{KL}=\left(  V_{KL},E\right)  $,
the discrete torus of size $K\times L$. We assume $K,L$ to be odd.

The dynamics of the network is a straightforward generalization of that of the
$C_{K}$ case. Again, the results and the proofs are similar, after we prove
the analog of Lemma \ref{L31}.

We can fix the labelling $\left(  1,1\right)  ,\left(  1,2\right)
,...,\left(  K,L\right)  $ on $\mathcal{T}_{KL}$; without loss of generality
we can take $j=\left(  1,1\right)  .$ But it is notationally more convenient
to introduce other coordinates on $\mathcal{T}_{KL}.$ Namely, we treat
$\mathcal{T}_{KL}$ as a product, $\mathcal{T}_{KL}=\left\{  -\frac{K-1}%
{2},...,\frac{K-1}{2}\right\}  \times\left\{  -\frac{L-1}{2},...,\frac{L-1}%
{2}\right\}  ,$ and, for our tagged element $j$, we now take $j=\left(
0,0\right)  $.

For every ordered pair $\left\{  u,v\right\}  , u, v \in V_{KL}$ such that
$\left\vert u-v\right\vert >1$, we define the set on next hop nodes from $u$
to $v$ as
\[
W\left(  u,v\right)  =\left\{  w\in V_{KL}:\left\vert u-w\right\vert
=1,\left\vert v-w\right\vert =\left\vert u-v\right\vert -1\right\}  .
\]
Clearly, $\left\vert W\left(  u,v\right)  \right\vert $ is either $1$ or $2.$
We define the \textit{random }map $T$ from the set of all ordered pairs
$\left\{  u,v\right\}  , u,v\in V_{KL}$ into the $V_{KL}\cup\left\{
\ast\right\}  $ by%
\[
T\left\{  u,v\right\}  =\left\{
\begin{array}
[c]{cc}%
w\text{ with probability }\frac{1}{\left\vert W\left(  u,v\right)  \right\vert
} &
\begin{array}
[c]{c}%
\text{if }\left\vert u-v\right\vert >1,\\
w\in W\left(  u,v\right)  ,
\end{array}
\\
\ast & \text{otherwise.}%
\end{array}
\right.
\]
In case $T\left\{  u,v\right\}  =w$ we say that we have a transit of a
customer through $w.$

Let $D:V_{KL}\rightarrow V_{KL}$ be an arbitrary map. { Let
\[
p_{KL}=\frac{1}{\left\vert S_{KL}\right\vert }\sum_{\pi\in S_{KL},u\in V_{KL}%
}\mathbb{P}\left(  T\left\{  \pi\left(  u\right)  ,D\left(  u\right)
\right\}  =\pi\left(  w\right)  \right)  .
\]
The proof of the following lemma is forwarded to the appendix. This lemma is
also a corollary of Lemma \ref{L55} below. }

\begin{lemma}
\label{lem:tor} For $K,L\geq3$
\[
p_{KL}=\frac{KL-5}{KL}.
\]

\end{lemma}

Hence, the following analogue of Theorem \ref{T3} holds.

\begin{theorem}
For each $\lambda>0$, $\beta>0$ and for each $K$ and $L$ such that
$KL>K^{\ast}=5/\lambda$, the process ${\mathcal{M}_{K,L}}$ is transient.
\end{theorem}

\subsection{{Regular graphs}}

Let us recall that a graph $G=\left(  V,E\right)  $ is called $g$-regular if
every vertex has $g$ edges adjacent to it. In this section we show that the
instability result established above actually holds for all connected and
$g$-regular graphs $G$. Clearly, the graphs $C_{K}$ and $\mathcal{T}_{KL}$ are
simple instances of such graphs.

Let us define $p_{G}$ analogously to the definition of {$p_{KL}$ in Lemma
\ref{lem:tor}:
\begin{equation}
p_{G}=\frac{1}{K!}\sum_{\pi\in S_{K},i\in V}\mathbb{P}\left\{  T\left\{
\pi\left(  i\right)  ,D\left(  i\right)  \right\}  =\pi\left(  j\right)
\right\}  , \label{332}%
\end{equation}
} where $K=\left\vert V\right\vert .$

If $G$ is connected and $g$-regular, then the symmetric nearest neighbor
random walk on it has the uniform measure as its unique stationary state.
Hence, lemmas \ref{L4}, \ref{L11} and \ref{L12} hold for $G.$ So the only step
needed is the following generalization of Lemma \ref{L31}:

\begin{lemma}
\label{L55}
\begin{equation}
p_{G}=\frac{\left\vert V\right\vert -\left(  g+1\right)  }{\left\vert
V\right\vert }. \label{88}%
\end{equation}

\end{lemma}

\begin{proof}
Let us label by $i=1,2,...,\left\vert V\right\vert $ the nodes and the servers
of our network. Suppose that server $i$ is initially located at node $i$ and
at node $\pi_{t}\left(  i\right)  $ at time $t$, where $\pi_{t}\in
S_{\left\vert V\right\vert }$ is a random permutation. Let $i,j$ be two
indices. We want to compute the {stationary transit rate from server $i$ to
$j$, assuming server $i$ has an infinite backlog of customers}. For the
transit event to happen, it is necessary that the two nodes $\pi_{t}\left(
i\right)  $ and $\pi_{t}\left(  j\right)  $ be neighbors. The fraction of time
it is the case is equal to $\frac{g}{\left\vert V\right\vert -1}$ (in the
stationary regime). If it does happen, then there are two options. The first
is that customer $c,$ served at $\pi_{t}\left(  i\right)  ,$ leaves the
network. This happens with probability $\frac{\left(  g+1\right)  }{\left\vert
V\right\vert }.$ The complementary event has probability $\frac{\left\vert
V\right\vert -\left(  g+1\right)  }{\left\vert V\right\vert }.$ In this case,
customer $c$ jumps to one of the $g$ neighboring nodes of $\pi_{t}\left(
i\right)  .$ Since server $\pi_{t}\left(  j\right)  $ can be at any of these
$g$ nodes with probability $\frac{1}{g},$ independently of the destination of
$c,$ the probability that $c$ will land on $\pi_{t}\left(  j\right)  $ is
$\frac{1}{g}.$ Since there are $\left\vert V\right\vert -1$ servers different
from server $j,$ we have:%
\[
p_{G}=\left(  \left\vert V\right\vert -1\right)  \frac{g}{\left\vert
V\right\vert -1}\frac{\left\vert V\right\vert -\left(  g+1\right)
}{\left\vert V\right\vert }\frac{1}{g}=\frac{\left\vert V\right\vert -\left(
g+1\right)  }{\left\vert V\right\vert }.
\]

\end{proof}

Hence, the following theorem holds:

\begin{theorem}
For each $\lambda>0$, $\beta>0$, and for each $K>K^{\ast}=(g+1)/\lambda$, the
Markov process ${\mathcal{M}_{G}}$ is transient.
\end{theorem}

\subsection{General graphs}

Analogously, the following result holds for a connected graph $G$ that is not
regular. Denote by $g$ the maximum degree of vertices $v\in V$ and let $K=|V|$.

\begin{theorem}
For each $\beta>0$, the process ${\mathcal{M}_{G}}$ as well as each mean-field
processes ${\mathcal{M}_{G^{N}}}$, $N=1,2,\dots$, is transient whenever
$K>K^{\ast}=(g+1)/\lambda$.
\end{theorem}

The only difference in the proof is that the equality $\left(  \ref{88}%
\right)  $ is replaced by the estimate%
\[
p_{G}\geq\frac{\left\vert V\right\vert -\left(  g+1\right)  }{\left\vert
V\right\vert }.
\]
We omit the proof. 

\section{{Mean-field infinite networks}}

\label{FT} { This section is focused on the mean-field version of certain
infinite networks. We first consider the networks $\left(  \mathbb{Z}%
^{1}\right)  ^{N}$ as $N\rightarrow\infty$. This leads to a NLMP on
$\mathbb{Z}^{1}$, which we study using the methodology introduced in
\cite{BRS} for a more general setting. This approach is then generalized to
Cayley graphs of discrete groups. }

\subsection{Non-linear Markov processes on $\mathbb{Z}^{1}$}

\label{ss}

In this section we consider the limit of the network $\left(  \mathbb{Z}
^{1}\right)  ^{N}$ as $N\rightarrow\infty,$
and more precisely the stationary distributions of this limiting network. We
focus on translation-invariant distributions, where invariance is w.r.t.
translations on $\mathbb{Z}$.

In this case, the NLMP is a dynamical system acting on measures $\mu$ on the
state of a queue.
The state of a queue $q_{v}$ can be identified with the sequence of (signed)
distances between the position of server $v$ and the destinations $D\left(
c_{i}\right)  $ of its customers. So the state becomes a finite integer-valued
sequence $\mathcal{N}\equiv\left\{  n_{1} ,...,n_{l};n_{i}\in\mathbb{Z}%
^{1}\right\}  $, where $l\geq0$ is the length of the queue $q_{v}.$

The rate of arrivals of transit customers,
leading, say, from the state $\left[  n_{1},...,n_{l-1}\right]  $ to the state
$\left[  n_{1},...,,n_{l-1},n_{l}\right]  $, is then a function of the integer
$n_{l}$ only, that we will denote by $\nu_{n_{l}}$.
With the above notation, we thus have
\begin{equation}
\nu_{k}\equiv\nu_{k}\left(  \mu\right)  =\left\{
\begin{array}
[c]{cc}%
\sum_{\mathcal{N}}\mu\left(  k+1,\mathcal{N}\right)   & \text{if }k>0,\\
\sum_{\mathcal{N}}\mu\left(  k-1,\mathcal{N}\right)   & \text{if }k<0,\\
0 & \text{if }k=0.
\end{array}
\right.  \label{12}%
\end{equation}
In what follows we look only for states $\mu$ which have symmetric rates
$\nu_{k}$, namely such that
\begin{equation}
\nu_{k}=\nu_{-k},\label{13}%
\end{equation}
for all $k\in\mathbb{Z}$. We also assume (for the sake of simplicity) that the
destination of a customer arriving at a node is this very same node.

The following result leverages the methodology developed in \cite{BRS}. It
provides a functional equation for fixed points of the NLMP. Each such fixed
point is a stationary regime of the NLMP.

Since the evolution of the NLMP is described by an infinite-dimensional
dynamical system in a space of probability measures (see Appendix), there
might exist non-trivial invariant sets of this evolution in the space of
probability measures (not just fixed points) Hence, other non-trivial
stationary measures might exist. The existence of non-trivial attractors for
the NLMP is discussed in \cite{RSV}. In the present paper, we restrict
ourselves to fixed points. The notation is that of the finite network case.

\begin{theorem}
\label{thmfu} { Under the foregoing assumptions, each fixed point $\mu$ of the
NLMP satisfy the functional equation:}
\begin{align}
&  \mu\left(  n_{1},...,n_{l-1}\right)  \left[  \nu_{n_{l}}(\mu)+\lambda
\delta\left(  n_{l},0\right)  \right]  -\mu\left(  n_{1},...,n_{l}\right)
\left(  {\sum_{k\ne0}\nu_{k}(\mu)+\lambda}\right) \nonumber\\
&  +\sum_{k}\mu\left(  k,n_{1},...,n_{l}\right)  -\mu\left(  n_{1}%
,...,n_{l}\right)  \mathbb{I}_{l\neq0}\label{14}\\
&  +\beta\left[  \mu\left(  n_{1}+1,...,n_{l}+1\right)  +\mu\left(
n_{1}-1,...,n_{l}-1\right)  -2\mu\left(  n_{1},...,n_{l}\right)  \right]
=0~.\nonumber
\end{align}

\end{theorem}

{ The proof is forwarded to the Appendix. This equation has a simple
interpretation. The term
\[
\mu\left(  n_{1},...,n_{l-1}\right)  \left[  \nu_{n_{l}}(\mu)+\lambda
\delta\left(  n_{l},0\right)  \right]
\]
is the arrival rate of the NLMP leading to state $[n_{1},...,n_{l}]$. The
term
\[
\mu\left(  n_{1},...,n_{l}\right)  \left(  \sum_{{k\ne0}} \nu_{k}(\mu
)+\lambda+ 2\beta\right)
\]
is the total rate out of $[n_{1},...,n_{l}]$. The term
\[
\sum_{k}\mu\left(  k,n_{1},...,n_{l}\right)  -\mu\left(  n_{1},...,n_{l}%
\right)  \mathbb{I}_{l\neq0}%
\]
is the departure rate leading to state $[n_{1},...,n_{l}]$. The term
\[
+\beta\left[  \mu\left(  n_{1}+1,...,n_{l}+1\right)  +\mu\left(
n_{1}-1,...,n_{l}-1\right)  \right]
\]
is the swap rate leading to state $[n_{1},...,n_{l}]$.}

As we will see later, Equations $\left(  \ref{12}\right)  -\left(
\ref{14}\right)  $ can have several solutions, one solution or no solution,
depending on the value of the parameter $\lambda.$ If $\mu$ is a solution of
Equations $\left(  \ref{12}\right)  -\left(  \ref{14}\right)  $ for some
$\lambda,$ then we denote by
\begin{equation}
\nu\left(  \mu\right)  =\sum_{{k\ne0}}\nu_{k}\left(  \mu\right)
\end{equation}
the rate, in state $\mu$, of the transit customers to every node and by
$\eta\left(  \mu\right)  $ the rate, in state $\mu$, of the total flow to
every node:
\begin{equation}
\eta\left(  \mu\right)  =\nu\left(  \mu\right)  +\lambda.
\end{equation}

\begin{theorem}
\label{z1} For every positive $\eta<1$ there exists a unique value
$\lambda\left(  \eta\right)  $ of the exogenous flow rate $\lambda$ and { a
unique measure $\mu_{\eta}$ on the set of queue states} satisfying Equations
$\left(  \ref{12}\right)  -\left(  \ref{14}\right)  $ with $\lambda
=\lambda\left(  \eta\right)  $, and such that $\eta\left(  \mu_{\eta}\right)
=\eta$.
\end{theorem}

\begin{proof}
{ In the mean-field limit, the total inflow rate to each node $v\in\mathbb{Z}$
is a Poisson process with the rate {$\eta$} for all $v$.} It is splitted
according to the possible destinations, {$v+h$, $h\in\mathbb{Z}$} of the
arriving customers. { The customers arriving to $v$ with destination $v+h$}
also form a Poisson process with rate $\nu_{h}$, so that we have
\[
{\eta}=\sum_{h\in{\mathbb{Z}}}\nu_{h},
\]
{with $\nu_{0}=\lambda$}. All these arrival Poisson processes are
independent.

Consider the random variable $\xi_{\eta}$, which is the total time the
customer spends in any given server in the stationary regime. It has
exponential distribution, which depends only on { $\eta=\sum_{k\ne0}\nu
_{k}+\lambda$,} which does not depend on the customer type. It is defined
uniquely by its expectation, which is $\mathbb{E}\left(  \xi_{\eta}\right)
=\left(  1-\eta\right)  ^{-1}.$

Consider now some tagged customer. Suppose it has type $k$ when arriving to
the tagged server. When it leaves the server, its type is changed to
$k+\tau_{\eta},$ where $\tau_{\eta}$ is an integer valued random variable.
This change happens due to the fact the server can move during the service
time of the tagged customer, i.e. to $\beta$-terms in $\left(  \ref{14}%
\right)  .$ By symmetry, $\mathbb{E}\left(  \tau_{\eta}\right)  =0.$ The
distribution of $\tau_{\eta}$ is the following. Consider a random walker
$W\left(  t\right)  $, living on $\mathbb{Z}^{1},$ which starts at $0$ (i.e.
$W\left(  0\right)  =0$) and which makes $\pm1$ jumps with rates $\beta.$ Then
$\tau_{\eta}=W\left(  \xi_{\eta}\right)  .$

The above observations lead to the following characterization of the rates
$\nu_{k}$ obtained from the stationary distribution of the following ergodic
Markov process on $\mathbb{Z}^{1}.$ Define the probability transition matrix
$P_{1}=\left\{  \pi_{st}\right\}  $ by $\pi_{st}=\Pr\left(  \tau_{\eta
}=s-t\right)  .$ Of course, this Markov chain on $\mathbb{Z}^{1}$ is not
positive recurrent since its mean drift is zero. Let $P_{2}$ be a second
Markov chain, with transition probabilities%
\[
\rho_{st}=\left\{
\begin{array}
[c]{cc}%
1 & \text{for }t=s-1,\ \ s\ge2,\\
1 & \text{for }t=0,\ s=1,0,-1,\\
1 & \text{for }t=s+1,\ s\le-2,\\
0 & \text{in other cases.}%
\end{array}
\right.
\]
The map $P_{2}$ is non-random map of $\mathbb{Z}_{1}$ into itself. Consider
the composition Markov chain, with transition matrix $Q=P_{1}P_{2}$. This
chain, which will be referred to as the single particle process below, is
positive recurrent (it has a drift towards the origin), and it hence has a
unique stationary distribution $q=\left\{  q_{k},k\in\mathbb{Z}^{1}\right\}
$. We take
\begin{equation}
\nu_{k}=\eta q_{k},\ k\neq0;\ \lambda=\eta q_{0}. \label{15}%
\end{equation}

Consider now the evolution of one single server queue with infinitely many
types $k\in\mathbb{Z}$ of customers, arriving to the queue according to
Poisson point processes with rates $\nu_{k}$, $k\neq0,$ and with rate
$\lambda$ for $k=0$, as defined by Equation $\left(  \ref{14}\right)  $.

Assume in addition that all customer types in the queue are incremented of one
unit according to a global exponential clock with rate $\beta$, and
decremented of one unit according to another independent global exponential
clock, also with rate $\beta$. This queuing process is an irreducible Markov
process, and since $\eta<1$, it is ergodic, so that it has a unique stationary
distribution. Denote by $\mu\left(  n_{1},...,n_{l}\right)  $ the stationary
probability of state $n_{1},...,n_{l}$ for this queue. By definition, these
probabilities satisfy $\left(  \ref{14}\right)  $.

In addition, it follows from the fact that $\{q_{k}\}$ is the steady state of
the Markov chain $Q$ that the rate $\nu_{k}$ (with $k>0,$ say) coincides with
the probability to find the queue in the state with the first customer having
the type $k+1.$ But this is exactly relation $\left(  \ref{12}\right)  $.

Note that the rates $\nu_{k}$ are defined in a unique way, once $\eta$ is
given, see above. As a result, the same uniqueness holds for the probabilities
$\mu\left(  n_{1},...,n_{l}\right)  $. This proves the existence and the
uniqueness statements of the theorem.
\end{proof}

We now state some properties of the function $\lambda\left(  \eta\right)  $ as
the parameter $\eta$ varies in $\left(  0,1\right)  .$

\begin{proposition}
\label{16} There is a $\lambda_{+}>0$ such that, for any positive
$\lambda<\lambda_{+}$, there are at least two different values $\eta=\eta
_{-}(\lambda)$ and $\eta=\eta_{+}(\lambda)$ satisfying the relation
$\lambda\left(  \eta\right)  =\lambda$ and such that $\eta_{-}(\lambda
)\rightarrow0$ and $\eta_{+}(\lambda)\rightarrow1$ as $\lambda\rightarrow0$.
\end{proposition}

\begin{proof}
Clearly, $\lambda\left(  \eta\right)  \rightarrow0$ as $\eta\rightarrow0.$ We
want to argue that $\lambda\left(  \eta\right)  \rightarrow0$ also when
$\eta\rightarrow1.$ Indeed, in this regime every customer spends more and more
time waiting in the queue, so for every $k$ the probability $\Pr\left(
\xi_{\eta}\leq k\right)  \rightarrow0$ as $\eta\rightarrow1.$ Therefore the
distribution of the random variable $\tau_{\eta}$ becomes more and more spread
out: for every $k,$ ${\mathbb{P}}\left(  \left\vert \tau\right\vert _{\eta
}\leq k\right)  \rightarrow0$ as $\eta\rightarrow1.$ Therefore the same
property holds for the stationary distribution $q,$ and the claim follows from
$\left(  \ref{15}\right)  $ and Proposition \ref{16}. In particular, this
implies that the equation (in $\eta$):
\[
\lambda\left(  \eta\right)  =a>0
\]
has at least two solutions for small $a$: $\eta_{-}$ close to 0 and $\eta_{+}$
close to $1$. This follows from the continuity of $\lambda\left(  \eta\right)
$.
\end{proof}

\subsubsection{{Computational illustration}}

{ Consider the evolution of the distance between the tagged customer and its
destination node in the above mean-field model. This is a continuous time
Markov chain on the non-negative integers with the following transition rates:
for $n>1$,
\begin{align*}
q(n,n+1) =\beta, \quad q(n,n-1) =\beta+\gamma,
\end{align*}
with $\gamma=1 -\eta$. This is because the tagged customer spends on a given
server an exponential time with parameter $\gamma$. Similarly,
\begin{align*}
q(1,2 ) =\beta, \quad q(1,0 ) =\beta, \quad q(1,\ast) =\gamma
\end{align*}
and
\begin{align*}
q(0,1 ) =2\beta, \quad q(0,\ast) =\gamma,
\end{align*}
where $\ast$ is absorbing. Let $T(n)$ be the mean time to absorption for a
customer at distance $n $ from its destination. The function $T(\cdot)$
satisfies the equations:
\begin{align*}
T(n)  &  =\frac{1}{2\beta+\gamma}+\frac{\beta}{2\beta+\gamma}T(n+1)+\frac
{\beta+\gamma}{2\beta+\gamma}T(n-1),\quad n>1,\\
T(1)  &  =\frac{1}{2\beta+\gamma}+\frac{\beta}{2\beta+\gamma}T(2) +\frac
{\beta}{2\beta+\gamma}T(0),\\
T(0)  &  =\frac{1}{2\beta+\gamma}+\frac{2\beta}{2\beta+\gamma}T(1).
\end{align*}
There is exactly one solution of these equations which behaves asymptotically
linearly as $n\rightarrow\infty$; all the other solutions behave
exponentially. This linear function is
\[
T(n)=\frac{1}{\gamma}\left(  n+\frac{\beta}{\gamma} \frac{2\beta+\gamma
}{3\beta+ \gamma}\right)  ,
\]
for all $n\ge1$. This in turn determines that
\[
T(0)= \frac1 \gamma+ \frac{\beta^{2}}{\gamma^{2}} \frac{ 2}{ 3\beta+\gamma}.
\]
Using the assumption that the destination of the tagged customer is the node
where it arrives, we get that the mean number of servers visited by the tagged
customer till absorption is
\begin{equation}
\mathbb{E}[N]= \gamma T(0)= 1 + \frac{2 \beta^{2}}{\gamma(3\beta+\gamma)}.
\end{equation}
The total rate in a queue hence satisfies the equation
\[
\eta=1-\gamma=\lambda\left(  1+ \frac{ \beta^{2} }{\gamma} \frac2
{3\beta+\gamma}\right)  .
\]
When $\beta$ is large, this boils down to the equation for $\nu=\eta-\lambda$
\[
\nu= \frac{2\lambda\beta}{3\gamma}= \frac{2\lambda\beta}{3(1-\lambda-\nu)},
\]
or equivalently to the equation
\[
3\nu^{2}-3\nu(1-\lambda) +2\lambda\beta=0.
\]
The discriminant is is positive if $\lambda<\frac4 3 \beta-4$ and in this
case, there are two roots
\begin{align*}
\nu^{+}  &  =\frac{1 -\lambda+\sqrt{(1-\lambda)^{2}-\frac8 3 \lambda\beta}}%
{2}\\
\nu^{-}  &  =\frac{1 -\lambda-\sqrt{(1-\lambda)^{2}-\frac8 3 \lambda\beta}}%
{2}.
\end{align*}
It is easy to see that under these conditions,
\[
0<\nu^{-}<\nu^{+}< 1 -\lambda,
\]
so that these are the two solutions given by the theory. }

This computational framework allows one to check the robustness of the
proposed framework to the specific assumptions made for mathematical
simplicity. One can for instance change the destination of a customer to be at
distance $d$ from the arrival node (rather than 0 here), or change the
absorption rule to be only when a service completes at the destination (rather
than the destination or one of the two neighbors here) and check by
computations of the same type that one still finds quite similar phenomena.

\subsection{{Non-linear Markov processes on Cayley graphs}}
This subsection extends the previous results in two ways.
First, $\mathbb{Z}^{1}$ is replaced by the Cayley graph of a countable group.
Second, we relax the assumption that the destination of an exogenous
customer arriving at some node is this same node.

\subsubsection{{Cayley networks}}

In this section we rewrite the equations which were studied in the last
subsection for the case of $\mathbb{Z}^{1}$ and we prove a theorem about the
structure of the stationary measures $\delta_{\mu}$ of the NLMP associated
with fixed points. Such measures $\mu$ will be called equilibria.

The underlying graph $G$ is assumed to be the Cayley graph of a countable
group (also denoted by $G$) with a finite generating set
\[
F=\left\{  g_{1},...,g_{k},g_{1}^{-1},...,g_{k}^{-1}\right\}  .
\]
Typical examples are $\mathbb{Z}^{d}$ or $\mathbb{T}^{d}$. The main theorem
will focus on the case of an infinite group.

 {
We suppose that the destination assignment rule, the jump direction, the jump
rates etc. are all $G$-invariant. The destination assignment rule is described by
$\Lambda=\left\{ \lambda_{h},h\in G\right\} $, where $\lambda_h$ denotes the rates of
external inflows to node $e$ of customers with address $h$.
At all other nodes the external flows have the same structure.
Let $\lambda=\sum_{h\in G} \lambda_{h}$.
Denote by $X_{G}$ the associated NLMP on $G$.
}

Consider a customer which finished its service at node $v,$ and assume it has
the neutral node $e\in G$ as its destination. Then it goes to the neighboring
node $b\left(  v\right)  $ which is closest to $e$ in graph distance on the
Cayley graph. If there are several such nodes, say $b_{1}\left(  v\right)
,\cdots,b_{R}\left(  v\right)  $, $R=R\left(  v\right)  $, then it chooses one
of them with probability $\frac{1}{R\left(  v\right)  }.$ We only look for
equilibria $\mu$ which are $G$-invariant. This means that, under $\mu$, the
rate $\nu$ of the Poisson flow of transit customers is the same at every node
$v$ and that the part of this flow consisting of customers with destination
$vh$ has rate $\nu_{h}$, which does not depend on $v$. This also means that,
in state $\mu$, the probability to have a queue of $l$ customers with
destinations $vh_{1} ,...,vh_{l}$ at node $v$ depends only on the string
$\left(  h_{1},...,h_{l}\right)  $ of elements of $G$. We denote this
probability by $\mu\left(  h_{1},...,h_{l}\right)  $.

{Denote
\begin{equation}
\eta_{h}=\nu_{h}+\lambda_{h}\quad\text{and}\quad\eta=\sum_{h\in G}\eta_{h}
\label{61}%
\end{equation}
}

The functional equations for the stationary measure now take the form {%
\[
0=-\mu(h_{1},\dots,h_{l})\left[  (1-\delta_{l=0})+\eta\right]  +\mu
(h_{1},\dots,h_{l-1}){\eta_{h_{l}}}%
\]
}%
\begin{equation}
+\sum_{h\in G}\mu(h,h_{1},\dots,h_{l})+\sum_{{g\in F}}\beta(\mu(h_{1}%
g,\dots,h_{l}g)-\mu(h_{1},\dots,h_{l})),\ l=0,1,2,...\ . \label{rz}%
\end{equation}
Our assumptions on the service discipline imply that the rates
 {$\nu_{h}$} are determined by the measure $\mu$ through the
following generalization of (\ref{12}):
\begin{equation}
{\nu_{h}}=\sum_{i=1}^{s}\frac{1}{R\left(  v_{i}\right)  }\sum_{\mathcal{N}}%
\mu\left(  v_{i},\mathcal{N}\right)  , \label{71}%
\end{equation}
where the inner summation is over all finite strings of nodes of $G.$ As in
Section \ref{ss}, our goal is to find  {the rates $\nu_{h}$,
$h\in G$}, satisfying (\ref{rz}) and (\ref{71}).


\subsubsection{Single particle process}

\label{322}

As in the special case of $\mathbb{Z}^{1}$ above, it will be useful to follow
the evolution of a single customer of the process $X_{G}$. In this subsection
we describe the associated continuous time Markov random walk.

Since the inflow rates at each node $v\in V$ in the process $X_{G}$ are
Poisson, in equilibrium all queue lengths are distributed geometrically and
customers have an exponentially distributed sojourn time $T$. If the total
arrival rate is $\eta<1$ and the service rate is $1$, then $T$ is
exponentially distributed with mean $H$ with $H=H\left(  \eta\right)
=1/(1-\eta)$.

We now describe the continuous-time Markov process ${\mathcal{B}}^{\eta}$ of a
an exogeneous particle that arrives to $e\in G,$ say, while the probability
distribution $d_{h}$ of its destination node $h$ is given by
\begin{equation}
d_{h}=\frac{\lambda_{h}}{\lambda}. \label{62}%
\end{equation}
The particle makes jumps of two kinds: random jumps due to the jumps of the
server harboring it (collectively with all customers sitting there) with rate
$\beta$, and directed individual jumps to a neighboring server that is closer
to the destination of the particle. The directed jumps happen at random times
(service event times) with inter-service intervals distributed exponentially
with mean $H$. If the particle is at distance $0$ or $1$ from its destination
and its service event happens, the particle dies (reaches the absorbing state).

Summarizing, we have the following:

\begin{claim}
For all jump rates $\beta$, probability distributions $\left\{  d_{h}\right\}
$ and total rate $\eta$, the address of the tagged customer is a continuous
time Markov random walk ${\mathcal{B}}^{\eta}$ on $G$.
\end{claim}

For some infinite graphs $G$ (say, for regular trees of degree $3$ or more)
the process ${\mathcal{B}}^{\eta}$ might be transient for all $\eta\in
\lbrack0,1)$, but this does not happen for $\beta$ small enough, and we only
consider the latter case below. Then the expected number of directed jumps of
the particle until absorption is finite for $\eta$ small enough. Let us denote
this expected number by $N(H(\eta))$. The function $N\left(  \cdot\right)  $
is a continuous increasing function, $N:{\mathbb{R}}_{+}\rightarrow
{\mathbb{R}}_{+}\cup\infty,$ such that $N(0)=\mathbb{E}_{\left\{
d_{h}\right\}  }\left(  \max\{\mathrm{dist}\left(  e,h\right)  ,1\}\right)  $
and $\lim_{\eta\rightarrow1}N(H(\eta))= \frac{\left\vert V\right\vert
}{\left\vert F\right\vert }$.

Some properties of the process $X_{G}$ can be retrieved from those of the
process ${\mathcal{B}}^{\eta}$. In particular, one can find the value of the
rate $\lambda$ of the external inflows to each node $v$ from the relation
\begin{equation}
\lambda=\lambda\left(  \eta\right)  =\frac{\eta}{N(H(\eta))}. \label{64}%
\end{equation}
Indeed, $\eta$, which is the load factor per station, is equal to the product
of the mean number of arrivals per unit time $\lambda$ and the mean number of
nodes visited by a typical customer.

The function $\lambda\left(  \eta\right)  $ thus defined is continuous on
$[0,1]$ and takes values $\lambda\left(  0\right)  =0$ and $\lambda\left(
1\right)  =|F|/|V|$. For infinite graphs $G$, we have $\lambda\left(
1\right)  =0$. Note that in some cases $\lambda\left(  \eta\right)  \equiv0$.
This happens, for instance, if $G$ is a tree of degree $g\geq3$ and $\beta$ is
large enough.

\subsubsection{Equilibria}

 {
Here is the generalization of Theorem \ref{z1}, where $G$, the
exogenous customers rates $\lambda_{h}$, $h\in G$, and the swap rate
$\beta$ are the basic data.
}

\begin{theorem}
 {
For every $0<\eta<1$, consider the processes $\mathcal{B}^{\eta}$
defined by the jump rate parameter $\beta$, the probability distribution
$\left\{ d_{h}=\frac{\lambda_{h}}{\lambda}\right\} $, and the rate $\eta$.
Assume that the lifetime of the random walk ${\mathcal{B}}^{\eta}$ has finite mean value.
For all $0<\eta<1$, there exist a unique $\lambda$ and a unique solution
$\mu_\eta(\cdot)$ of (\ref{rz})-(\ref{71})
such that the associated transit rates
$\{\nu_{h},h\in G\}$ satisfy the condition
$$\sum_{h\in G} \nu_{h}+\lambda_h=\eta.$$
}
\end{theorem}

\begin{proof}
The proof follows closely that of Theorem \ref{z1}. The random walker
$W\left(  t\right)  $ now lives on $\mathbb{\ }G$. It starts at $e$ and it
makes jumps from each node $v$ to one of the neighboring nodes $\left\{
vg_{1},...,vg_{k},vg_{1}^{-1},...,vg_{k}^{-1}\right\}  $ with rates $\beta.$
The matrix of transition probabilities $P_{1}$ is now defined via that random
walk. Again the corresponding Markov chain on $G$ is not positive recurrent.
Let $P_{2}$ be the Markov chain with transition probabilities
\[
\rho_{st}=\left\{
\begin{array}
[c]{cc}%
\frac{1}{R\left(  s\right)  } & \text{for }t=b_{j}\left(  s\right) \\
0 & \text{in other cases.}%
\end{array}
\right.
\]
Then we consider the composition Markov chain with transition matrix
{$Q=P_{1}P_{2}$.} The rest of the constructions proceeds in the same way.
\end{proof}

\subsection{Extension to non-exponential service times}

Let us describe a simple extension of the results of this section to the case
where the customer service times are i.i.d. with unit mean and bounded second
moment, again at the nodes of Cayley graph. Let us denote by $\xi$ the random
service time of a single customer and by $F_{\xi}$ its distribution function.

{The description of the corresponding NLMP, which will be denoted by
$X_{G,\xi}$, requires of course the extra variables -- namely, the amounts of
service already received by the customers becomes relevant, see \cite{BRS}.}

For all $G$-invariant stationary measures of $X_{G,\xi}$, each server receives
a total inflow which is a homogeneous Poisson point process of rate $\eta$, and
we can again decompose $\eta$ as {$\eta=\sum_{h\in G}\eta_{h}$}, as described
in the previous subsection, with $\eta_{h}$ being the total arrival rate of
customers with destination $h$ at server $e$. Each customer stays at each node
for a stationary sojourn time which is that of a $M/GI/1$ queue with
parameters $\eta$ and $\xi$. Additionally, all customers in this queue change
their locations simultaneously with rate $\beta$ as their server swaps its
position with the adjacent server.

Consider the stochastic process ${\mathfrak{B}}^{\eta,\xi}$ of the random walk
of a single particle over the Cayley graph of $G$ till its exit. This process
is analogous to the process ${\mathfrak{B}}^{\eta}$ in Section \ref{322}. The
only difference between these two processes is that the exponentially
distributed random variable $T$ (stationary sojourn time of the system
$M/M/1$) is replaced by $T_{\eta,\xi}$, the stationary sojourn time in the
queue $M/GI/1$ with i.i.d. service time with distribution $F_{\xi}$ {given by
the Pollaczek-Khinchine formula.}

As in Section \ref{322}, let us require that the triple $(\eta,\xi,\beta)$
ensures finiteness of the mean lifetime of a customer, denoted by $E(\eta
,\xi,\beta)$. For fixed $\xi$ and $\beta$, denote by $\bar{\eta}$ the right
endpoint of the maximal interval of values $\eta$ where the function
$E(\eta,\xi,\beta)$ is finite. Clearly, $\bar{\eta}$ depends continuously on
$\beta$ and, within the interval $[0,\bar{\eta})$, the function $E(\eta
,\xi,\beta)$ is continuous in $\eta$ and $\beta$.

{ }

\begin{theorem}
Consider the NLMP $X_{G,\xi}$. Assume that $E(\eta,\xi,\beta)$ is finite. Let
us fix a probability distribution $d_{h}=\lambda_{h}/\lambda$, $h\in G$, where
$\lambda=\sum_{h}\lambda_{h}.$ Then there is a unique value of $\lambda$ such
that the process $X_{G,\xi}$ has a unique $G$-invariant stationary distribution.
\end{theorem}

\begin{proof}
Let us fix an arbitrary triple $(u,v,w)$ of points of $G$. For the single
particle process {${\mathfrak{B}}^{\eta,\xi}$}, let us denote by $k(u,v,w)$
the mean number of direct jumps to node $v$ (only the jumps of the particle
are counted, not the jumps of servers) of the particle that is located at node
$u$ and has the address $w$. One can easily see that $\eta$ can be decomposed
as follows:
\begin{equation}
\eta=\sum_{u,w}\lambda_{v^{-1}w}k(u,v,w). \label{zv2}%
\end{equation}
Thanks to $G$-invariance, $\eta$ does not depend on $v$.
The right-hand side of this equation is finite as shown when
rewriting $k(u,v,w)$ as $k\left(e,vu^{-1},wu^{-1}\right)$, thanks to G-invariance.
So the right-hand side is finite as for 
${\mathfrak{B}}^{\eta,\xi}$
(the single particle random walk) the mean life time of the particle is finite.

Accordingly, for the $G$-invariant stationary distribution $X_{G,\xi}$ of the
NLMP, the rate of the homogeneous Poisson inflow to node $e$ of particles with
address $h$ can be written as
\begin{equation}
\eta_{h}=\lambda_{h}k(e,e,h)+\sum_{u\neq e}\lambda_{u^{-1}h}k(u,e,h).
\label{zv}%
\end{equation}
Also,
\[
\nu_{h}=\sum_{u\neq e}\lambda_{u^{-1}h}k(u,e,h).
\]
Since for a given distribution $\{d_{h}\}$the sum
\[
\lambda_{h}k(e,e,h)+\sum_{u\neq e}\lambda_{u^{-1}h}k(u,e,h)
\]
is a linear function of $\lambda$, for a given $\eta$, Equation (\ref{zv}) has
a unique solution $\{\lambda_{h},\nu_{h},h\in G\}$.

Hence we have a stationary process of arrival to node $e$ (or any other node)
of independent Poisson flows of different types $k$. The total rate of this
inflow is $\eta<1$ and the total queue at a given server is stationary. The
customers in this queue change their types to $k^{\prime}$ after an
exponentially distributed time, in the same manner as in the case of
exponentially distributed variable $\xi$.

The process which describes the queue is ergodic since it has renewal times
when the queue becomes empty. The product of stationary measures on the queues
over all the nodes is the required $G$-invariant measure for the NLMP
$X_{G,\xi}$.
\end{proof}

\begin{theorem}
If $\lambda>0$ is small enough, the NLMP on a translation invariant infinite
graph has at least two translation invariant equilibria.
\end{theorem}

\begin{proof}
It follows from the continuity of the function $E(\eta,\xi,\beta)$ in $\eta$
and $\beta$ and from the relations (\ref{zv}), (\ref{zv2}), that $\lambda$ is
a continuous function of $\eta$. Moreover, $\lambda\to0$ as $\eta\to0$.
Analogously, $\lambda\to0$ as $\eta\to\bar{\eta}$.
\end{proof}

\section{Comparing ${\mathbb{Z}}^{d}_{K}$ and ${\mathbb{Z}}^{d}$}

In the simplest case $d=1$ we have the following property.

\begin{theorem}
The circle ${\mathbb{Z}}^{1}_{K}$ is faster than the line ${\mathbb{Z}}^{1}$
in terms of absorption times.
\end{theorem}

\begin{proof}
First, we consider ${\mathbb{Z}}^{1}_{+}$ instead of ${\mathbb{Z}}^{1}$. Then
the circle ${\mathbb{Z}}^{1}_{K}$ can be associated with an interval $[0,K/2]$
within $Z^{1}_{+}$. The two processes are the same on this interval apart from
the endpoint $K/2$.

One may say that the process on ${\mathbb{Z}}^{1}_{+}$ is coupled with a
special process on $[0,K/2]$ as follows. As two processes are at $K/2$ and the
``line" process goes further to $K/2+1$, the ``circle" process waits at $K/2$
till the ``line" process returns to $K/2$. Hence, the absorption time may only increase.
\end{proof}

In the general case $d\ge1$ we can only say that in some sense the equilibria
on ${\mathbb{Z}}^{d}_{K}$ are close to the equilibrium on ${\mathbb{Z}}^{d}$
for $K$ large enough. Namely, the following assertion holds.

\begin{theorem}
For each $\eta\in\lbrack0,1]$,
\[
\lim_{K\rightarrow\infty}\lambda_{K}(\eta)=\lambda\left(  \eta\right)  .\text{
}%
\]

\end{theorem}

We omit the technical proof.

{ }

\section{Conclusions}

{ Let us conclude with a few observations on the physical meaning of our
models and their connections to earlier models of the networking literature. }

\subsection{Relationship between the models}

Let us first discuss the relationships between the replica 1 model, the
replica $N$ model and the replica $\infty$ model.

{ The replica-1 model and the replica-$N$ model describe two different
physical systems as illustrated by the following wireless communication
network setting: The replica-$N$ model features a system with $N$ frequency
bands and where each device has $N$ radios, one per frequency band, that it
can use simultaneously, for both transmission and reception. It then
implements a set of $N$ virtual FIFO queues of packets, one per band. At any
time, a device transmits the packet head of the queue to the appropriate
neighboring device (the one closest to packet destination) on that band. Upon
reception, the packet is loaded in one of the $N$ queues of the receiver
chosen at random. In this sense, the replica-$N$ model is \emph{implementable}%
, and the $1$-replica method is just the special case with one frequency band.
}

{ Let us now explain in what sense the replica-$\infty$ model (which is
computational but non implementable) tells us something on the replica-$N$
model (which is implementable but not computational), in spite of the fact
that the latter is always unstable whereas the former admits stationary
regimes: the replica-$N$ network has a metastable state, which lives for
longer and longer times as $N$ grows, and which is well described by the
minimal solution of the replica-$\infty$ model. }

 {
\subsection{Relationship with the Gupta and Kumar models}
}  {
Our setting is close to
that of Gupta and Kumar \cite{GK}.
The latter also involves a network with $N$
nodes using hop by hop relaying through nearest
neighbors with traffic from every node to a node
chosen at random among all others. The initial model has no node mobility.
The main finding of \cite{GK} is that the
arrival rate that can be sustained on any such node is order $1/\sqrt{N}$
when nodes are located in the Euclidean plane.
Indeed, since nearest neighbors are at distance $1/\sqrt{N}$,
if packet destination is at distance order 1, then every packet has to
travel order $\sqrt{N}$ hops, so that every node
has to relay a total flow of order $\lambda \sqrt{N}$
(as seen by analyzing the load brought per queue). Hence the
capacity (maximum value of $\lambda$ for which the system is stable) is
order $1/\sqrt{N}$.
}

 {
The case with mobility was then considered by
\cite{GT} and \cite{Shah}.
The main finding is that the above scaling of the capacity
is in fact {\em improved} by node mobility. At first glance,
our result suggests that mobility does worse
than absence of mobility (for instance, in the setting of the
present paper, where packet destination at distance order 1 for
graph distance,
the static network has a stability region order 1, whereas
the mobile one has a stability region that tends to 0 when $N$
tends to infinity).
}

 {
Let us explain why there is no contradiction in fact:
our model can be extended to the setting
where the initial distance to destination is order $\sqrt{N}$
for graph distance (rather
than 0 or order 1), which will then be comparable in terms of load to that of
the Gupta and Kumar setting in the Euclidean plane.
In the absence of mobility, our model has a stability region
of order $1/\sqrt{N}$, as in Gupta and Kumar, whereas
in the presence of mobility, as shown by our analysis, a customer
brings a load to order $N$ queues, which is much worse than in
the case without mobility. The reason lies in the use of the
FIFO discipline as a constraint. If the scheduling in queues were more
adaptive than FIFO, we could do exactly as \cite{GT}, namely
keep the packets in their arrival station and wait for this station to
be close to a packet destination to schedule the latter.
Then the load per queue is order 1. Of course,
delay is terrible (return to 0 of a random walk) as in \cite{GT}. In conclusion,
there is no contradiction with this literature.
}

 {
These connections lead to the following observations:
\begin{enumerate}
\item Our model offers a new microscopic view of
this class of problems which complements both the Gupta and Kumar \cite{GK}
and the Grossglauser and Tse scaling \cite{GT} laws.
This microscopic view, describes what happens in a typical
queue, and opens a new quantitative line of thoughts (through mean fields)
for this class of problems.
\item Our model does not sacrifice delays (like
Grossglauser and Tse)
and finds one possible compromise between delay and rate in line with
the ideas discussed by \cite{Shah}.
A study of adaptive scheduling (e.g. priority to customers close to destination)
might lead to innovative solutions to this class of questions.
\end{enumerate}}

\section{Appendix}


\subsection{Proof of Lemma \ref{lem:tor}}

Again we partition the group $S_{KL}$ into cosets. Each coset $A_{\pi}$ has
now $K\left(  K-1\right)  L\left(  L-1\right)  $ elements. With every
permutation $\pi$ of the $KL$ points of the discrete torus $\mathcal{T}_{KL}$
we first include in $A_{\pi}$ all its `2D cyclic moves', i.e. permutation
$\pi$ followed by an arbitrary shift of $\mathcal{T}_{KL}$; there are $KL$ of
them. Also, with each permutation $\tilde{\pi}$ from $A_{\pi}$ we include in
$A_{\pi}$ all $\left(  K-1\right)  \left(  L-1\right)  $ permutations which
are obtained from $\tilde{\pi}$ by performing a pair of `independent
restricted cyclic moves': one of them cyclically permutes all the sites of the
meridian of the point $\pi\left(  0,0\right)  ,$ and another cyclically
permutes all the sites belonging to the parallel of the point $\pi\left(
0,0\right)  .$ There are $\left(  K-1\right)  \left(  L-1\right)  $ such
independent restricted cyclic moves. All other $\left(  K-1\right)  \left(
L-1\right)  $ points of the torus $\mathcal{T}_{KL},$ as well as the point
$\pi\left(  0,0\right)  ,$ stay fixed during these independent restricted
cyclic moves. The idea behind this definition is to ensure the following
property: suppose we know for the permutation $\pi$ that a point $\pi\left(
k^{\prime},l^{\prime}\right)  $ belongs to the parallel of the point
$\pi\left(  0,0\right)  ,$ while the point $\pi\left(  k^{\prime\prime
},l^{\prime\prime}\right)  $ belongs to the meridian of $\pi\left(
0,0\right)  .$ Then for any three different points $a,b^{\prime}%
,b^{\prime\prime}\in\mathcal{T}_{KL},$ such that $b^{\prime}$ belongs to the
parallel of $a,$ and $b^{\prime\prime}$ belongs to the meridian of $a,$ there
exists exactly one permutation $\bar{\pi}\in A_{\pi},$ such that $\bar{\pi
}\left(  0,0\right)  =a,$ $\bar{\pi}\left(  k^{\prime},l^{\prime}\right)
=b^{\prime}$ and $\bar{\pi}\left(  k^{\prime\prime},l^{\prime\prime}\right)
=b^{\prime\prime}.$

Now we fix one class $A_{\pi},$ and compute the number of transits to node
$\bar{\pi}\left(  0,0\right)  $ for all $\bar{\pi}\in A_{\pi}.$ Without loss
of generality and in order to simplify the notations we consider only the case
when $\pi$ is the identity $e\in S_{KL}.$ We denote the permutations from
$A_{e}$ by the letter $\varkappa.$

Clearly, the transits from node $\varkappa\left(  k,l\right)  $ to
$\varkappa\left(  0,0\right)  $ can happen only if either $k$ or $l$ are $0.$
So let us fix some integer $k\in\left\{  -\frac{K-1}{2},...,\frac{K-1}%
{2}\right\}  ,$ $k\neq0,$ and let us count the number of possible transits
from $\varkappa\left(  k,0\right)  $ to $\varkappa\left(  0,0\right)  $ while
$\varkappa$ runs over $A_{e}.$ Without loss of generality we can assume that
the destination $D\left(  k,0\right)  =\left(  0,0\right)  .$

As we said already, as $\varkappa$ runs over $A_{e},$ the node $\varkappa
\left(  0,0\right)  $ can be anywhere on the torus $\mathcal{T}_{KL}.$ The
node $\varkappa\left(  k,0\right)  $ can be then anywhere on the parallel of
$\varkappa\left(  0,0\right)  .$ If $\varkappa\left(  0,0\right)  =\left(
0,0\right)  \left(  =D\left(  k,0\right)  \right)  ,$ then no transit to
$\varkappa\left(  0,0\right)  $ can happen, independently of the location of
$\varkappa\left(  k,0\right)  .$ The same is true when $\varkappa\left(
0,0\right)  =\left(  x,y\right)  $ with $x=\pm\frac{K-1}{2},$ $-\frac{L-1}%
{2}\leq y\leq\frac{L-1}{2}.$

For every of the $\left(  K-3\right)  $ locations $\left(  x,0\right)
\in\mathcal{T}_{KL}$ of the node $\varkappa\left(  0,0\right)  $ -- namely,
for $x=-\frac{K-1}{2}+1,...,-1,+1,+2,...\frac{K-1}{2}-1$ we have one transit
per location (or, more precisely, $L-1$ transits per location, due to the
restricted cyclic moves along the meridian).

For every of the location $\left(  0,y\right)  $ of the node $\varkappa\left(
0,0\right)  $ -- namely, for $y=-\frac{L-1}{2},...,-1,+1,+2,...,\frac{L-1}{2}$
we have $2\times\frac{1}{2}=1$ transits per location (or, again more
precisely, $L-1$ transits per location), since there can be two transit
events, each with probability $\frac{1}{2}.$

For any other remaining location of the node $\varkappa\left(  0,0\right)  $
-- and there are $\left(  K-3\right)  \left(  L-1\right)  $ of them, we get
$\frac{1}{2}$ of transit per location (more precisely, $\frac{L-1}{2}$
transits per location).

Summarizing, we have totally $\left[  \left(  K-3\right)  +\left(  L-1\right)
+\frac{1}{2}\left(  K-3\right)  \left(  L-1\right)  \right]  \left(
L-1\right)  $ transits from the node $\varkappa\left(  k,0\right)  $ to the
node $\varkappa\left(  0,0\right)  $, as $\varkappa$ runs over $A_{e}.$ And
there are $\left(  K-1\right)  $ such nodes.

All in all, we have%
\begin{align*}
&  \left[  \left(  K-3\right)  +\left(  L-1\right)  +\frac{1}{2}\left(
K-3\right)  \left(  L-1\right)  \right]  \left(  K-1\right)  \left(
L-1\right) \\
&  +\left[  \left(  L-3\right)  +\left(  K-1\right)  +\frac{1}{2}\left(
L-3\right)  \left(  K-1\right)  \right]  \left(  L-1\right)  \left(
K-1\right)
\end{align*}
transits, so the probability in question is given by%
\begin{align*}
&  \frac{\left(  K-3\right)  +\left(  L-1\right)  +\frac{1}{2}\left(
K-3\right)  \left(  L-1\right)  +\left(  L-3\right)  +\left(  K-1\right)
+\frac{1}{2}\left(  L-3\right)  \left(  K-1\right)  }{KL}\\
&  =\frac{3\left(  K-3\right)  +3\left(  L-3\right)  +\left(  K-3\right)
\left(  L-3\right)  +4}{KL}\\
&  =\frac{3K+3L+\left(  K-3\right)  \left(  L-3\right)  -14}{KL}\\
&  =\frac{KL-5}{KL}.
\end{align*}

\subsection{Proof of Theorem \ref{thmfu}}

Using the results (and notation) of \cite{BRS}, we get that the NLMP is the
evolution of the measure $\otimes\mu_{v}$ on the states (queues $q_{v}$) of
the servers at the nodes $v\in\mathbb{Z}^{1},$ given by the equations%

\begin{equation}
\frac{d}{dt}\mu_{v}\left(  q_{v},t\right)  =\mathcal{A}+\mathcal{B}%
+\mathcal{C}+\mathcal{D}+\mathcal{E} \label{001}%
\end{equation}
with
\begin{equation}
\mathcal{A}=-\frac{d}{dr_{i^{\ast}\left(  q_{v}\right)  }\left(  q_{v}\right)
}\mu_{v}\left(  q_{v},t\right)  \label{06}%
\end{equation}
the derivative along the direction $r\left(  q_{v}\right)  $ (in our case,
since we assume exponential service times with rate 1, we have $\frac
{d}{dr_{i^{\ast}\left(  q_{v}\right)  }\left(  q_{v}\right)  }\mu_{v}\left(
q_{v},t\right)  =\mu_{v}\left(  q_{v},t\right)  $),
\begin{equation}
\mathcal{B}=\delta\left(  0,\tau\left(  e\left(  q_{v}\right)  \right)
\right)  \mu_{v}\left(  q_{v}\ominus e\left(  q_{v}\right)  ,t\right)  \left[
\sigma_{tr}\left(  q_{v}\ominus e\left(  q_{v}\right)  ,q_{v}\right)
+\sigma_{e}\left(  q_{v}\ominus e\left(  q_{v}\right)  ,q_{v}\right)  \right]
\label{201}%
\end{equation}
where $q_{v}$ is created from $q_{v}\ominus e\left(  q_{v}\right)  $ by the
arrival of $e\left(  q_{v}\right)  $ from $v^{\prime}$, and $\delta\left(
0,\tau\left(  e\left(  q_{v}\right)  \right)  \right)  $ takes into account
the fact that if the last customer $e\left(  q_{v}\right)  $ has already
received some amount of service, then he cannot arrive from the outside;
\begin{equation}
\mathcal{C}=-\mu_{v}\left(  q_{v},t\right)  \sum_{q_{v}^{\prime}}\left[
\sigma_{tr}\left(  q_{v},q_{v}^{\prime}\right)  +\sigma_{e}\left(  q_{v}%
,q_{v}^{\prime}\right)  \right]  ,
\end{equation}
which corresponds to changes in queue $q_{v}$ due to customers arriving from
other servers and from the outside (in the notations of $\left(
\ref{21}\right)  ,$ $\sigma_{e}\left(  q_{v},q^{v}\oplus w\right)
=\lambda_{v,w}$);
\begin{equation}
\mathcal{D}=\int_{q_{v}^{\prime}:q_{v}^{\prime}\ominus C\left(  q_{v}^{\prime
}\right)  =q_{v}}d\mu_{v}\left(  q_{v}^{\prime},t\right)  \sigma_{f}\left(
q_{v}^{\prime},q_{v}^{\prime}\ominus C\left(  q_{v}^{\prime}\right)  \right)
-\mu_{v}\left(  q_{v},t\right)  \sigma_{f}\left(  q_{v},q_{v}\ominus C\left(
q_{v}\right)  \right)  ,
\end{equation}
where the first term describes the situation where the queue $q_{v}$ arises
after a customer was served in a queue $q_{v}^{\prime}$ (longer by one
customer), and $q_{v}^{\prime}\ominus C\left(  q_{v}^{\prime}\right)  =q_{v},$
while the second term describes the completion of service of a customer in
$q_{v}$;
\begin{equation}
\mathcal{E}=\sum_{v^{\prime}\text{n.n.}v}\beta_{vv^{\prime}}\left[
\mu_{v^{\prime}}\left(  q_{v},t\right)  -\mu_{v}\left(  q_{v},t\right)
\right]  , \label{202}%
\end{equation}
where the $\beta$-s are the rates of exchange of the servers.

For the convenience of the reader we repeat the equation $\left(
\ref{001}-\ref{202}\right)  $ once more:
\begin{align}
&  \frac{d}{dt}\mu_{v}\left(  q_{v},t\right)  =-\frac{d}{dr_{i^{\ast}\left(
q_{v}\right)  }\left(  q_{v}\right)  }\mu_{v}\left(  q_{v},t\right)
\nonumber\\
&  +\delta\left(  0,\tau\left(  e\left(  q_{v}\right)  \right)  \right)
\mu_{v}\left(  q_{v}\ominus e\left(  q_{v}\right)  \right)  \left[
\sigma_{tr}\left(  q_{v}\ominus e\left(  q_{v}\right)  ,q_{v}\right)
+\sigma_{e}\left(  q_{v}\ominus e\left(  q_{v}\right)  ,q_{v}\right)  \right]
\nonumber\\
&  -\mu_{v}\left(  q_{v},t\right)  \sum_{q_{v}^{\prime}}\left[  \sigma
_{tr}\left(  q_{v},q_{v}^{\prime}\right)  +\sigma_{e}\left(  q_{v}%
,q_{v}^{\prime}\right)  \right]  +\int_{q_{v}^{\prime}:q_{v}^{\prime}\ominus
C\left(  q_{v}^{\prime}\right)  =q_{v}}d\mu_{v}\left(  q_{v}^{\prime}\right)
\sigma_{f}\left(  q_{v}^{\prime},q_{v}^{\prime}\ominus C\left(  q_{v}^{\prime
}\right)  \right) \label{233}\\
&  -\mu_{v}\left(  q_{v}\right)  \sigma_{f}\left(  q_{v},q_{v}\ominus C\left(
q_{v}\right)  \right)  +\sum_{v^{\prime}\text{n.n.}v}\beta_{vv^{\prime}%
}\left[  \mu_{v^{\prime}}\left(  q_{v}\right)  -\mu_{v}\left(  q_{v}\right)
\right]  ~.\nonumber
\end{align}
Compared to the setting of \cite{BRS}, we have the following simplifications:

\begin{enumerate}
\item The graph $G$ is the lattice $\mathbb{Z}^{1};$

\item All customers have the same class;

\item The service time distribution $\eta$ is exponential, with the mean value
$1;$

\item The service discipline is FIFO;

\item The exogenous customer $c$ arriving to node $v$ has for destination the
same node $v$; inflow rates at all nodes are equal to $\lambda;$

\item The two servers at $v,v^{\prime},$ which are neighbors in $\mathbb{Z}%
^{1}$ exchange their positions with the same rate $\beta\equiv\beta
_{vv^{\prime}};$
\end{enumerate}

The equation for the fixed point then becomes:%
\begin{align*}
&  0=\mu_{v}\left(  q_{v}\ominus e\left(  q_{v}\right)  \right)  \left[
\sigma_{tr}\left(  q_{v}\ominus e\left(  q_{v}\right)  ,q_{v}\right)
+\sigma_{e}\left(  q_{v}\ominus e\left(  q_{v}\right)  ,q_{v}\right)  \right]
\\
&  -\mu_{v}\left(  q_{v}\right)  \sum_{q_{v}^{\prime}}\left[  \sigma
_{tr}\left(  q_{v},q_{v}^{\prime}\right)  +\lambda\right]  + \sum
_{q_{v}^{\prime}:q_{v}^{\prime}\ominus C\left(  q_{v}^{\prime}\right)  =q_{v}%
}\mu_{v}\left(  q_{v}^{\prime}\right) \\
&  -\mu_{v}\left(  q_{v}\right)  \mathbb{I}_{q_{v}\neq\varnothing}%
+\sum_{v^{\prime}=v\pm1}\beta\left[  \mu_{v^{\prime}}\left(  q_{v}\right)
-\mu_{v}\left(  q_{v}\right)  \right]  ~.
\end{align*}
The proof is concluded when using the fact that queue $q_{v}$ can in this
setting be identified with the sequence of destinations $D\left(
c_{i}\right)  $ of its customers.

\end{document}